\documentclass[10 pt]{article}

\usepackage[utf8]{inputenc}

\usepackage{amssymb}
\usepackage{amsthm}
\usepackage{amsmath}
\usepackage{anysize}
\usepackage{url}
\usepackage{color}
\usepackage{graphicx}
\usepackage{overpic}
\usepackage{subcaption}
\usepackage{xcolor}

\usepackage{bm}

\usepackage{a4wide,url}

\usepackage[colorlinks]{hyperref}

\usepackage{paralist}

\usepackage{comment}

\newcommand{\abs}[1]{\left| #1 \right|}

\newcommand{\norm}[1]{\left\lVert#1\right\rVert}

\newcommand{\N}{\mathbb{N}}

\newcommand{\R}{\mathbb{R}}

\newcommand{\C}{\mathbb{C}}

\newcommand{\Cd}{\mathbb{C}^d}

\newcommand{\Oab}{\mathcal{O}_\alpha^\beta}

\newcommand{\poincarec}{C_{P}}

\newcommand{\RN}[1]{\uppercase\expandafter{\romannumeral #1\relax}}

\graphicspath{{pics/}}

\newtheorem{theorem}{Theorem}
\newtheorem{remark}[theorem]{Remark}

\newtheorem{definition}[theorem]{Definition}
\newtheorem{proposition}[theorem]{Proposition}
\newtheorem{lemma}[theorem]{Lemma}
\newtheorem{corollary}[theorem]{Corollary}

\newtheorem{innercustomthm}{Theorem}

\newenvironment{customthm}[1]
  {\renewcommand\theinnercustomthm{#1}\innercustomthm}
  {\endinnercustomthm}
\numberwithin{theorem}{section}
\numberwithin{equation}{section}

\date{}

\begin{document}
\title{Stable Gabor phase retrieval for multivariate functions}

\author{Philipp Grohs\footnote{Faculty of Mathematics, University of Vienna, Oskar Morgenstern Platz 1, 1090 Vienna, Austria and Research Platform DataScience@UniVie, University of Vienna, Oskar Morgenstern Platz 1, 1090 Vienna, Austria} and Martin Rathmair\footnote{Faculty of Mathematics, University of Vienna, Oskar Morgenstern Platz 1, 1090 Vienna, Austria}}

\maketitle
\abstract{In recent work [P. Grohs and M. Rathmair. Stable Gabor Phase Retrieval and Spectral Clustering. Communications on Pure and Applied Mathematics (2018)] the instabilities 
of the Gabor phase retrieval problem, i.e., the problem of reconstructing a function $f$ from its spectrogram $|\mathcal{G}f|$, where
$$
\mathcal{G}f(x,y)=\int_{\R^d} f(t) e^{-\pi|t-x|^2} e^{-2\pi i t\cdot y} dt, \quad x,y\in \R^d,
$$
have been completely classified in terms of the disconnectedness of the spectrogram. 
These findings, however, were crucially restricted to the onedimensional case ($d=1$) and therefore not relevant for many practical applications.\\
In the present paper we not only generalize the aforementioned results to the multivariate case but also significantly improve on them.
Our new results have comprehensive implications in various applications such as ptychography, a highly popular method in coherent diffraction imaging.
}

\section{Introduction}
\subsection{Motivation}\label{sec:motivation}
\emph{Phase retrieval} in its most general formulation is concerned with the reconstruction of a signal $f\in \mathcal{B}$ with $\mathcal{B}$ a Banach space from phaseless linear measurements
\begin{equation}\label{eq:phaselessgeneral}
\mathcal{A}f:=\left(|\phi_\omega(f)|\right)_{\omega\in \Omega},
\end{equation}
where $\Phi=(\phi_\omega)_{\omega\in \Omega} \subset \mathcal{B}'$, the dual of $\mathcal{B}$.\\
Problems of this kind appear in a vast number of physical applications. 
The most prominent example being coherent diffraction imaging\cite{Miao2015530,shechtman2015phase}, where one seeks to recover a function from phaseless Fourier type measurements, so called diffraction patterns.
Further applications include radar\cite{jaming1999phase}, astronomy\cite{dainty1987phase}, audio\cite{waldspurger2015wavelet} and quantum mechanics\cite{paul2008introduction} to mention only a few.\\
Usually the measurement vectors $\Phi$ are such that 
$
f\mapsto \left(\phi_\omega(f)\right)_{\omega\in \Omega}
$
is nicely invertible, meaning that reconstructing $f$ would not be a significant problem if the phases of the measurements were available.
The removal of phase, however, not only involves the loss of a huge amount of information but also renders the problem nonlinear.
It is therefore notoriously difficult to even decide whether a concrete phase retrieval problem is well posed, 
i.e., whether the measurements $\mathcal{A}f$ uniquely and stabily determine the underlying signal $f$ in the following sense:\\ 
{\bf Uniqueness:} Is the mapping $f\mapsto \mathcal{A}f$ injective up to the identification $f\sim e^{i\alpha}f$ for $\alpha\in \R$? \\
{\bf Stability:} What is the qualitative behaviour of the local stability constant, i.e., the smallest number $c(f)$ such that
\begin{equation}\label{eq:stabilitygeneral}
d_{\mathcal{B}}(g,f):= \inf_{|a|=1} \|g-a f\|_{\mathcal{B}} \le c(f) d'(\mathcal{A}g,\mathcal{A}f) \quad \forall g\in \mathcal{B},
\end{equation}
where $d'$ denotes a suitable metric on the measurement space?\\
At this point we would like to draw the attention to our very recent article together with 
Sarah Koppensteiner\cite{grohs2019mathpr} where, among other things, our current understanding of uniqueness and stability for phase retrieval is summarized.\\

The present paper is concerned with the study of stability when the measurements arise from the so called \emph{Gabor} transform.
The Gabor transform is just the short time Fourier transform with Gaussian window.
\begin{definition}\label{def:gabortrafo}
 The Gabor transform of $f\in L^2(\R^d)$ is defined by 
 \begin{equation*}
  \mathcal{G}f(x,y)=\int_{\R^d} f(t) e^{-\pi|t-x|^2} e^{-2\pi i t\cdot y} dt, \quad x, y \in \R^d.
 \end{equation*} 
 By duality the definition can be extended to the dual of the space of Schwartz functions, i.e., the space of tempered distributions denoted by $\mathcal{S}'(\R^d)$.
\end{definition}
\begin{figure}
\centering
 \includegraphics[width=0.6\textwidth]{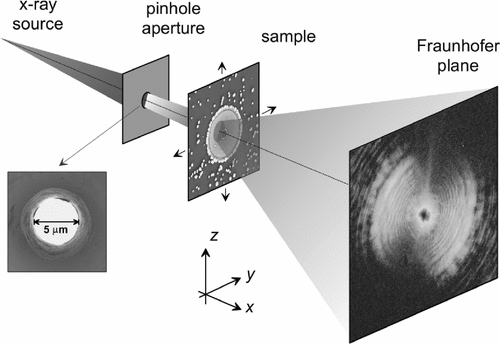}
 \caption{Schematic setup of ptychographical experiment. Image taken from \cite{rodenburg2007hard}.}
 \label{fig:ptycho}
\end{figure}
Note that by choosing the measurement vectors to be time-frequency shifts of the Gaussian
$$
\phi_\omega(t)= e^{-\pi|t-x|^2}e^{-2\pi i y\cdot t}, \quad \omega=(x,y)\in \Omega\subset\R^{2d}
$$
the Gabor transform fits right into our setting, i.e., in that case $\mathcal{A}f$ as defined in Equation \eqref{eq:phaselessgeneral} coincides with $|\mathcal{G}f|$.\\
The Gabor transform can be interpreted as localization of $f$ at $x$ followed by Fourier transform
$$
\mathcal{G}f(x,y)=\mathcal{F}\left(fe^{-\pi |\cdot - x|^2}\right)(y).
$$
Thus, for a two dimensional object, represented by $f\in L^2(\R^2)$, the magnitude of the Gabor transform $|\mathcal{G}f(x,\cdot)|$ describes the diffraction pattern of 
the localization of $f$ at $x$.
Hence the Gabor transform perfectly mimics the concept of \emph{ptychography}, a popular and highly successful approach in coherent diffraction imaging based on the idea that multiple diffraction patterns of
one and the same object are generated by illuminating different sections of the object seperately in order to introduce redundancy, cf. Figure \ref{fig:ptycho}.\\
%
%
\subsection{Related work}
We will now briefly discuss results regarding stability properties of phase retrieval in infinite dimensional spaces.
All results into this direction are fairly recent without exception.\\
First of all, inconveniently, phase retrieval in infinite dimensions is severely ill-posed as it can never be uniformly stable, in the sense that 
$c(f)$ in \eqref{eq:stabilitygeneral} can never be uniformly bounded, i.e., $\sup_{f\in \mathcal{B}} c(f)=+\infty$, under very general assumptions on $\mathcal{B}, \Phi$ and $d'$ \cite{cahill2016phase,grohsstab}.
This means that there are functions $f$ and $\tilde{f}$ such that the respective measurements $\mathcal{A}f$ and $\mathcal{A}\tilde{f}$ are arbitrarily close while $f$ and $\tilde{f}$ are not similar at all.
In that case $f$ is informally refered to as an 'instability'.\\
Note that this behaviour stands in stark contrast to the finite dimensional situation where uniqueness readily implies global stability.\\
Furthermore, if the infinite dimensional space $\mathcal{B}$ is approximated by an increasing sequence $\mathcal{B}_1\subset \mathcal{B}_2 \subset \ldots\subset \mathcal{B}$ 
of finite dimensional subspaces with $\dim(\mathcal{B}_n)=n$ then the global stability constant of the restricted problem, i.e., the smallest number $c_n$ such that
\begin{equation*}
\inf_{|a|=1} \|g-a f\|_{\mathcal{B}} \le c_n d'(\mathcal{A}g,\mathcal{A}f) \quad \forall f,g\in \mathcal{B}_n,
\end{equation*}
may degenerate exponentially in $n$ \cite{cahill2016phase,grohsstab}.\\

For the concrete example of Gabor phase retrieval explicit instabilities can be constructed by taking two functions $f_1$ and $f_2$ which have time-frequency support on two disjoint domains, meaning that
$\mathcal{G}f_1$ and $\mathcal{G}f_2$ are essentially supported on disjoint domains.
In that case the spectrograms of $f_+:=f_1+f_2$ and $f_-:=f_1-f_2$ approximately coincide since
$$
|\mathcal{G}(f_1\pm f_2)|^2 = |\mathcal{G}f_1|^2 \pm 2 \Re \left(\mathcal{G}f_1 \overline{\mathcal{G}f_2}\right) + |\mathcal{G}f_2|^2 \approx |\mathcal{G}f_1|^2 + |\mathcal{G}f_2|^2.
$$
For details see \cite{aifarigrohsinstab}.
Qualitatively all instabilities obtained in this way are of the same type, namely their spectrograms essentially live on a domain which is a disconnected set in the time-frequency plane.
A quantitative concept that precisely captures this kind of disconnectedness is provided by the so called Cheeger constant, which plays a prominent role in Riemannian geometry \cite{cheeger} 
and spectral graph theory \cite{Chung:1997}.
\begin{definition}\label{def:cheegerconst}
Let $\Omega\subset \R^d$ be a domain and let $w$ be a nonnegative, continuous function on $\Omega$.
Then the \emph{Cheeger constant} of $w$ is defined by 
$$
h(w,\Omega):=\inf_{C \in\mathcal{C}} \frac{\int_{\partial C \cap \Omega} w}{\min\{\int_C w, \int_{\Omega\setminus C} w\}},
$$
where $\mathcal{C}:=\{C\subset \Omega \text{ open}:~ \partial C \cap \Omega \text{ is smooth}\}$.
\end{definition}
\begin{remark}\label{rem:cheegerconst}
 A small Cheeger constant $h(w,\Omega)$ indicates that the domain $\Omega$ can be partitioned into $C\subset \Omega$ and $\Omega\setminus C$ such that the weight $w$ is small along the seperating boundary and, 
 at the same time, that both $C$ and $\Omega\setminus C$ approximately carry the same amount of $L^1$-energy w.r.t. $w$.
 In that sense $w$ then consists of multiple components; we say that $w$ is of the diconnected type.\\
 If on the other hand $w$ is concentrated on a connected domain -- the Gaussian being a prime example -- a partition which accomplishes both objectives simultaneously does not exist. 
 See Figure \ref{fig:disconnect} where two concrete examples are considered.
\end{remark}
\begin{figure}
  \includegraphics[width=0.33\textwidth]{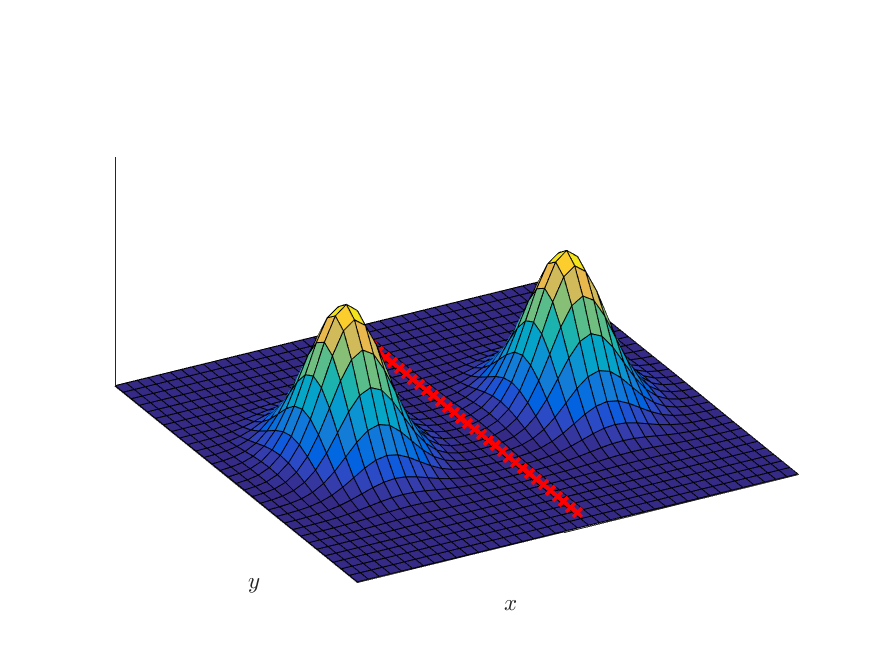}
  \includegraphics[width=0.33\textwidth]{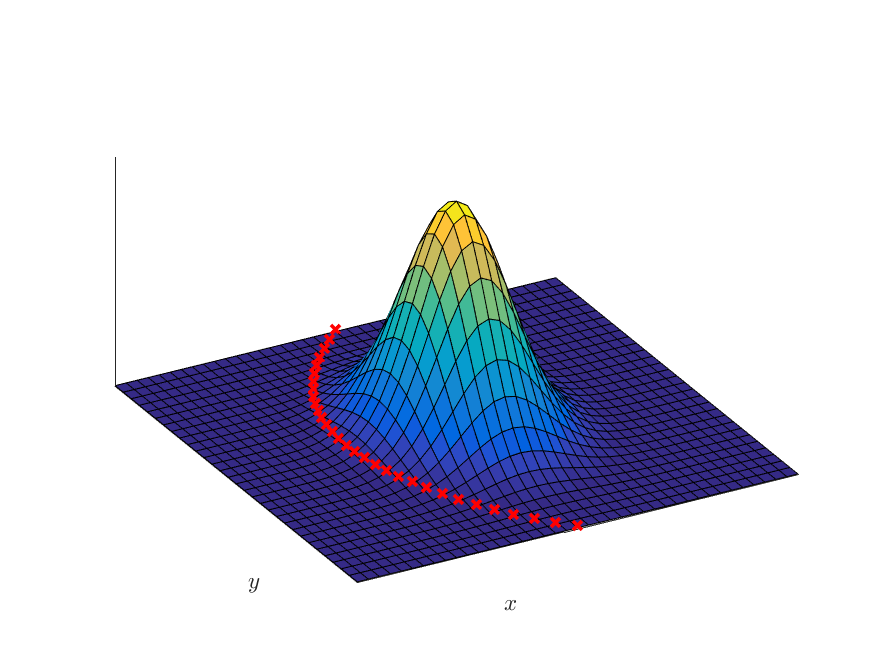}
  \includegraphics[width=0.33\textwidth]{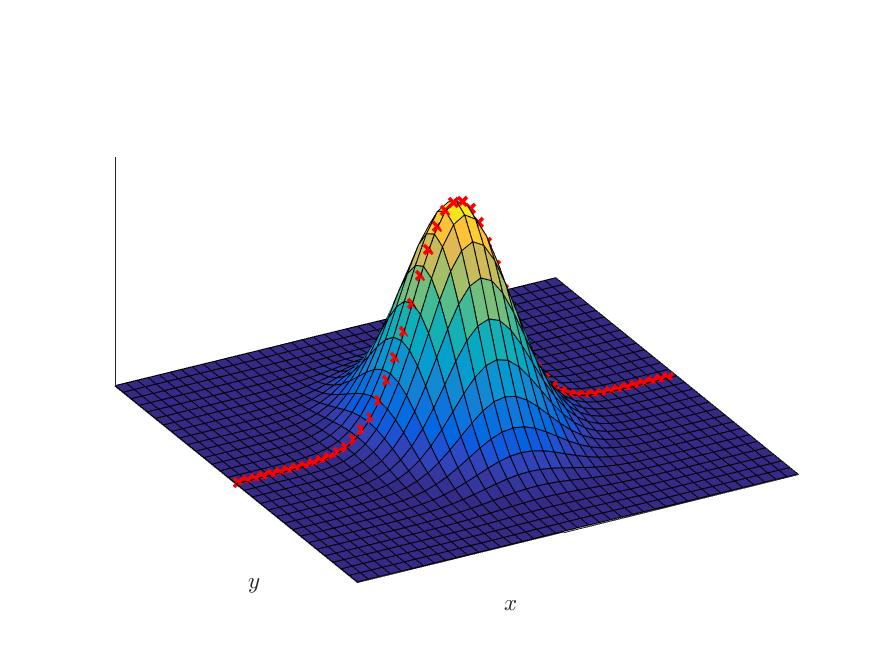}
  \caption{Comparison of possible partitions of the domain in the disconnected(left) and connected case(center and right).}
  \label{fig:disconnect}
\end{figure}
Based upon the work of one of the authors and his collaborators\cite{alaifari2016stable}, the paramount discovery in our preceding article\cite{stablegaborpr} is that 
the local stability constant $c(f)$ for phase retrieval from Gabor magnitudes of univariate functions 
can be essentially controlled by the reciprocal of the Cheeger constant of the spectrogram of $f$, i.e. by $h(|\mathcal{G}f|,\Omega)^{-1}$. 
This insight nicely complements the picture as it reveals that all instabilities are of the disconnected type.
However the results are fundamentally restricted to the onedimensional setting.
\subsection{Contribution}
The main contribution of this article is that we establish the connection between Cheeger constant and stability of Gabor phase retrieval obtained in \cite{stablegaborpr}
for multivariate signals of arbitrary dimension.
The function spaces best suited for our analysis are the so called \emph{modulation spaces}.
\begin{definition}
For $p\ge 1$ the modulation spaces are defined by
\begin{equation}\label{eq:defmodspace}
 \mathcal{M}^p(\R^d)=\left\{f\in \mathcal{S}'(\R^d):~\mathcal{G}f\in L^p(\R^{2d})\right\},
\end{equation}
with induced norm $\|f\|_{\mathcal{M}^p(\R^d)}=\|\mathcal{G}f\|_{L^p(\R^{2d})}$.
\end{definition}
With the modulation spaces at hand we can now state a special case of our main result, Theorem \ref{thm:gaborstability}.
\begin{customthm}{A}\label{thm:stabapp}
Suppose that $f\in\mathcal{M}^1(\R^d)$ is such that $|\mathcal{G}f|$ has a global maximum at the origin and let $q>2d$.
Then for all $g\in \mathcal{M}^1(\R^d)$ it holds that
\begin{multline}\label{eq:stabestapp}
\inf_{|a|=1} \|g-af\|_{\mathcal{M}^1(\R^d)} \lesssim (1+h(|\mathcal{G}f|,\R^{2d})^{-1}) \\
\cdot \left(
\||\mathcal{G}g|-|\mathcal{G}f|\|_{W^{1,1}(\R^{2d})} + \|(1+|\cdot|^{2d+2})\left(|\mathcal{G}g|-|\mathcal{G}f|\right)\|_{L^q(\R^{2d})}
\right)
\end{multline}
where the implicit constant depends on $d$ and $q$ only.
\end{customthm}
Note however that Theorem \ref{thm:gaborstability} is way more general as it also covers the case 
where the phase of the Gabor transform on a domain $\Omega\subsetneq\R^{2d}$ is to be reconstructed given $|\mathcal{G}f(\omega)|$, $\omega\in \Omega$.\\
Our results have an immediate impact for substantial applications.
One of them being ptychography -- as briefly discussed in Section \ref{sec:motivation} -- where the object of interest is represented by a function of more than one variable.
Theorem \ref{thm:stabapp} identifies precisely for which ptychographic measurements reconstruction is possible in a stable manner.\\
We would like to stress that the results in the present paper are not merely a straight forward generalization of our results from \cite{stablegaborpr} to higher dimensions.
The proof methods have undergone several modifications which not only makes for a slicker reading but also leads to notably improved results: 
Our earlier analysis only guaranteed estimates as in \eqref{eq:stabestapp} where the implicit constant mildly depended on $f$.
This dependency is now entirely removed, i.e., the stability constant can indeed be controlled in terms of the reciprocal of the Cheeger constant.\\
Our proof methods draw upon techniques from various fields of mathematics such as functional analysis, Riemannian geometry, complex analysis in several variables and potential theory.  
The second main emphasis lies on the study of certain quantities, such as the logarithmic derivative, of entire functions satisfying specific growth restrictions.
The results we derive into this direction play a vital role in our analysis of Gabor phase retrieval.
These results do not only serve as an auxiliary intermediate step but are rather interesting in their own right, and therefore merit to be highlighted at this stage:
\begin{customthm}{B}\label{thm:estlogderiv}
 Suppose that $G$ is an entire function on $\C^d$ such that $\sup_{|z|\le r} |G(z)|\le |G(0)| e^{\alpha r^\beta}$ for all $r>0$.
 Then it holds that
 $$
 \|G'/G\|_{L^1(B_r)} \lesssim r^{2d+\beta-1}, \quad r>0,
 $$
 where the implicit constant depends on $d,\alpha$ and $\beta$ but not on $G$.
\end{customthm}
Theorem B is a special case of Theorem \ref{thm:main}.

\subsection{Preliminaries and Notation}
In the present paper we will constantly identify $\Cd$ with $\R^{2d}$ via 
$$
(z_1,\ldots,z_d)=(x_1+iy_1,\ldots,x_d+iy_d)\quad  \leftrightarrow \quad (x_1,y_1,\ldots,x_d,y_d);
$$
accordingly, a domain $\Omega$ in $\Cd$ can be considered as a domain in $\R^{2d}$ and vice versa.
We will denote balls of radius $r$ centered at $u$ by
$B_r(u):=\{z: |z-u|<r\}$; if $u=0$ we will just write $B_r$.\\
A complex valued function $F$ on a domain $\Omega\subset \Cd$ is differentiable at $u\in \Omega$ if it is differentiable w.r.t. to $x_1,y_1,\ldots,x_d,y_d$.
In that case we write 
$$
\nabla F(u)=\left(\frac{\partial}{\partial x_1} F(u), \frac{\partial}{\partial y_1} F(u),\ldots, \frac{\partial}{\partial x_d} F(u),\frac{\partial}{\partial y_d} F(u)\right)^T.
$$
We will also use the so called \emph{Wirtinger derivatives} defined by
$\frac{\partial}{\partial z_j} = \frac12 \left(\frac{\partial}{\partial x_j} - i \frac{\partial}{\partial y_j} \right)$ for $1\le j \le d$.
If $F$ is complex differentiable at $u$ we will occasionally use the notation
$$
F'(u)=\left(\frac{\partial}{\partial z_1} F(u),\ldots,\frac{\partial}{\partial z_d} F(u) \right)^T
$$
We denote the space of holomorphic functions on a domain $\Omega\subset \C^d$ by $\mathcal{O}(\Omega)$ 
and the space of meromorphic functions, i.e., functions that locally coincide with the quotient of two holomorphic functions, by $\mathcal{M}(\Omega)$.\\
For a measureable, nonnegative function $w$ on $\Omega$ and $1\le p <+\infty$ we denote the weighted Lebesgue space by $L^p(\Omega,w)$ consisting of all measureable functions $F$ on $\Omega$ such that 
\begin{equation}\label{def:wtdLpnorm}
\|F\|_{L^p(\Omega,w)}:=\left( \int_\Omega |F|^p w\right)^{1/p} <+\infty.
\end{equation}
In the unweighted case, i.e., if  $w\equiv 1$, we will just write $L^p(\Omega)$ and $\|\cdot \|_{L^p(\Omega)}$ instead.
Note that \eqref{def:wtdLpnorm} also makes sense for vector valued functions $F$ by understanding $|F|$ as the euclidean length.\\
The Sobolev norms are defined by
\begin{equation*}
 \|\cdot\|_{W^{1,p}(\Omega)} := \|\cdot \|_{L^p(\Omega)} + \|\nabla\cdot \|_{L^p(\Omega)}.
\end{equation*}
\section{Stability and Cheeger constants}
\subsection{A first stability result}
This section will unveil our key mechanism for deriving stability estimates for phase retrieval under the general assumption that the quotient of two measurements is meromorphic.
This mechanism relies on the interplay between Poincar\'e inequalities, as defined next, and complex analysis.
\begin{definition}\label{def:poincare}
 Let $\Omega\subset \Cd$ be a domain equipped with a nonnegative, integrable weight $w$ and let $1\le p < +\infty$.
 We say that $\Omega$ supports a \emph{Poincar\'e inequality} if there exists a finite constant $C$ such that 
 \begin{equation}\label{eq:poincareineq}
  \inf_{c\in \C} \|F-c\|_{L^p(\Omega,w)} \le C \|\nabla F\|_{L^p(\Omega,w)} \quad \text{for all } F\in \mathcal{M}(\Omega)\cap L^p(\Omega,w).
 \end{equation}
The smallest possible constant in \eqref{eq:poincareineq} is called the \emph{Poincar\'e constant} of $\Omega$ and denoted by $\poincarec(\Omega,w,p)$.
\end{definition}
\begin{remark}
 Note that in the defining inequality of Definition \ref{def:poincare} functions are restricted to be meromorphic.
 This is certainly nonstandard but precisely the right concept for our purposes.
 Due to the famous \emph{Lavrentiev phenomenon}\cite{zhikov1995lavrentiev}, which states that smooth functions need not necessarily be dense in weighted Sobolev spaces,
 a Poincar\'e inequality of the type defined above does not necessarily imply a Poincar\'e inequality in the usual sense.
\end{remark}
The main result of this section provides an upper bound for the distance between two measurements whose quotient is assumed to be meromorphic 
in terms of an expression which only depends on the moduli of the two measurements.
\begin{theorem}\label{thm:firststab}
 Let $\Omega\subset \Cd$ be a domain and let $1\le p < +\infty$.
 Suppose that $F_1,F_2\in L^p(\Omega)$ are such that their quotient $F_2/F_1$ is meromorphic.
 Then it holds that 
 \begin{multline}\label{est:firststab}
  \inf_{|c|=1} \|F_2-cF_1\| \le   \||F_2|-|F_1|\|_{L^p(\Omega)} \\
  +2^{3/2} \poincarec(\Omega,|F_1|^p,p) \left( \|\nabla |F_1|-\nabla |F_2|\|_{L^p(\Omega)} 
  + \left\| \frac{\nabla |F_1|}{|F_1|} \left(|F_1|-|F_2|\right) \right\|_{L^p(\Omega)} \right)
 \end{multline}
\end{theorem}
Inequality \eqref{est:firststab} is already quite close to a stability estimate of the desired mould.
If we neglect the logarithmic derivative $\nabla|F_1|/|F_1|$ for a moment, it states that --provided that $\poincarec(\Omega,|F_1|^p,p)$ is moderately small-- the distance between two measurements is comparable to the 
distance of the respective moduli, as measured in the Sobolev norm.\\
The remainder of this section is devoted to proving Theorem \ref{thm:firststab}.
A key role in the proof will be played by the fact that for holomorphic functions $F$ the local variation of $|F|$ coincides with the local variation of $F$ up to a factor. 
\begin{lemma}\label{lem:keylemma}
 Suppose that $F$ is holomorphic at a point $w\in \Cd$.
 Then it holds that $$|\nabla |F|(w)|=2^{-1/2}|\nabla F(w)|=|F'(w)|.$$
\end{lemma}
\begin{proof}
 Let us first assume that $d=1$. We split $F=u+iv$ into its real and imaginary part.
 At every point where $F$ is differentiable it holds that
 $$
 \nabla |F| = \nabla (u^2+v^2)^{1/2} = \frac12 \frac{2u \nabla u + 2v \nabla v}{|F|} = \frac{u \nabla u + v \nabla v}{|F|}
 $$
 and therefore 
 $$
 |\nabla |F||^2 = \frac{u^2 |\nabla u|^2 + v^2 |\nabla v|^2 +2 uv\left(\nabla u \cdot \nabla v\right)}{u^2+v^2}.
 $$
 Since $F$ is assumed to be holomorphic at $w$ the Cauchy-Riemann equations hold at $w$.
 In particular it holds that $|\nabla v(w)|^2=|\nabla u(w)|^2$ and that $\nabla u(w)\cdot \nabla v(w)=0$.
 Thus, it follows that $|\nabla |F|(w)|^2=|\nabla u(w)|^2$.\\
 On the other hand we obtain
 \begin{equation}\label{eq:nablaF}
 |\nabla F(w)|^2 = |\nabla u(w) + i \nabla v(w)|^2 = |\nabla u(w)|^2 + |\nabla v(w)|^2 =2 |\nabla u(w)|^2, 
 \end{equation}
 where we used again that $|\nabla v(w)|^2=|\nabla u(w)|^2$.
 Therefore the first identity $|\nabla |F|(w)|=2^{-1/2} |\nabla F(w)|$ holds true.\\
 For the second identity note that since $F$ is holomorphic at $w$ it holds that $F'(w)=\frac{\partial}{\partial z} F(w)= \frac{\partial}{\partial x} F(w)$.
 Thus, by making use of Cauchy-Riemann equations again and by \eqref{eq:nablaF}, we have that
 \begin{equation*}
|F'(w)|^2=u_x^2(w)+v_x^2(w)=|\nabla u(w)|^2=2^{-1}|\nabla F(w)|^2
 \end{equation*}
and therefore that $|F'(w)|=2^{-1/2} |\nabla F(w)|$.\\
 
The general case $d>1$ follows from the univariate case: 
Note that for any $1\le j \le d$ the mapping $\C \ni z\mapsto F(w_1,\ldots,w_{j-1}, z , w_{j+1},\ldots,w_d)$ is holomorphic at $w_j$.
Then by the first part it holds that
\begin{equation*}
 \begin{aligned}
  |\nabla |F|(w)|^2 &= \sum_{j=1}^d  \left[\left(\frac{\partial}{\partial x_j} |F|(w) \right)^2 + \left(\frac{\partial}{\partial y_j} |F|(w) \right)^2\right]\\
		    &= \sum_{j=1}^d 2^{-1} \left[ \left(\frac{\partial}{\partial x_j} F (w) \right)^2 + \left(\frac{\partial}{\partial y_j} F(w) \right)^2\right] =2^{-1} |\nabla F(w)|^2.
 \end{aligned}
\end{equation*}
Similarly we obtain that
\begin{equation*}
  |F'(w)|^2 = \sum_{j=1}^d \left| \frac{\partial}{\partial z_j} F(w)\right|^2
	    =  \sum_{j=1}^d 2^{-1}\left[ \left(\frac{\partial}{\partial x_j} F (w) \right)^2 + \left(\frac{\partial}{\partial y_j} F(w) \right)^2\right] =2^{-1} |\nabla F(w)|^2,
\end{equation*}
which completes the proof.
 \end{proof}
With Lemma \ref{lem:keylemma} at hand we are ready to prove the main result of this section:
\begin{proof}[Proof of Theorem \ref{thm:firststab}]
In a first, preperatory step we show that the constraint $|c|=1$ in the distance 
$$
\inf_{|c|=1} \|F_2-cF_1\|_{L^p(\Omega)}
$$
can effectively be dropped if the unsigned measurements are close, cf. Inequality \eqref{est:unconstrainedstep1}.
The distance term without the constraint on $c$ will be controlled in the second step by making use of Poincar\'e's inequality as well as Lemma \ref{lem:keylemma}.\\

{\bf Step 1: Getting rid of the constraint}\\
 For $\epsilon>0$ let $c_\epsilon \in \C$ be such that 
 \begin{equation}\label{eq:defceps}
 \|F_2-c_\epsilon F_1\|_{L^p(\Omega)} \le \inf_{c\in \C} \|F_2-c F_1\|_{L^p(\Omega)} + \epsilon.
 \end{equation}
 Note that by continuity $c_\epsilon$ can always be choosen to be nonzero.
 By making use of triangle inequality we estimate
 \begin{equation}\label{est:unconstrained1}
  \inf_{|c|=1} \|F_2-c F_1\|_{L^p(\Omega)} \le \|F_2-\frac{c_\epsilon}{|c_\epsilon|} F_1\|_{L^p(\Omega)} \le \|F_2-c_\epsilon F_1\|_{L^p(\Omega)} + \|c_\epsilon F_1-\frac{c_\epsilon}{|c_\epsilon|} F_1\|_{L^p(\Omega)}.
 \end{equation}
Furthermore the last term can be bounded by
\begin{equation}\label{est:unconstrained2}
\begin{aligned}
 \|c_\epsilon F_1-\frac{c_\epsilon}{|c_\epsilon|} F_1\|_{L^p(\Omega)} &=\| |c_\epsilon F_1|-|F_1|\|_{L^p(\Omega)}\\
					 &\le \| |c_\epsilon F_1|-|F_2|\|_{L^p(\Omega)} + \||F_2|-|F_1|\|_{L^p(\Omega)}\\
					 &\le \| c_\epsilon F_1-F_2\|_{L^p(\Omega)} + \||F_2|-|F_1|\|_{L^p(\Omega)}.
\end{aligned}
\end{equation}
Combining \eqref{est:unconstrained1} and \eqref{est:unconstrained2}, together with \eqref{eq:defceps} yields that 
$$
\inf_{|c|=1} \|F_2-c F_1\|_{L^p(\Omega)} \le 2 \inf_{c\in \C} \|F_2-c F_1\|_{L^p(\Omega)} + 2 \epsilon + \| |F_2|-|F_1|\|_{L^p(\Omega)}.
$$
Since $\epsilon>0$ was arbitrary it holds that
\begin{equation}\label{est:unconstrainedstep1}
\inf_{|c|=1} \|F_2-c F_1\|_{L^p(\Omega)} \le 2 \inf_{c\in \C} \|F_2-c F_1\|_{L^p(\Omega)} + \| |F_2|-|F_1|\|_{L^p(\Omega)}.
\end{equation}

{\bf Step 2: Bound for the unconstrained distance}\\
First we rewrite for arbitrary $c\in \C$
$$
\|F_2-c F_1\|_{L^p(\Omega)} = \|F_2/F_1-c\|_{L^p(\Omega,|F_1|^p)}.
$$
Note that by assumption the quotient $F_2/F_1$ is meromorphic, that $\|F_2/F_1\|_{L^p(\Omega,|F_1|^p)} = \|F_2\|_{L^p(\Omega)}<+\infty$ and that the weight $|F_1|^p$ is integrable due to the assumption that $F_1\in L^p(\Omega)$.
Therefore Poincar\'e's inequality can be applied and we obtain
\begin{equation}\label{est:poincare}
 \begin{aligned}
  \inf_{c\in\C} \|F_2-c F_1\|_{L^p(\Omega)} &= \inf_{c\in \C} \|F_2/F_1-c\|_{L^p(\Omega,|F_1|^p)}\\
				       &\le \poincarec(\Omega,|F_1|^p,p) \left\|\nabla \frac{F_2}{F_1}\right\|_{L^p(\Omega,|F_1|^p)}\\
				       &=2^{1/2} \poincarec(\Omega,|F_1|^p,p) \left\|\nabla \left|\frac{F_2}{F_1}\right|\right\|_{L^p(\Omega,|F_1|^p)}
 \end{aligned}
\end{equation}
where the last equality follows from Lemma \ref{lem:keylemma}, since $F_2/F_1$ is holomorphic in $\Omega$ with the exception of a set of measure zero.
Using that 
$$
\nabla \left|\frac{F_2}{F_1}\right| = |F_1|^{-2} \cdot \left( |F_1|\nabla |F_2|- |F_2|\nabla |F_1|\right)
$$
almost everywhere in $\Omega$ we estimate
\begin{equation}\label{est:gradientmod}
 \begin{aligned}
  \left\|\nabla \left|\frac{F_2}{F_1}\right|\right\|_{L^p(\Omega,|F_1|^p)} &= \left\| |F_1|^{-2} \cdot \left( |F_1|\nabla |F_2|- |F_2|\nabla |F_1|\right) \right\|_{L^p(\Omega,|F_1|^p)}\\
				    &\le \left\| \frac{|F_1|\nabla |F_2|- |F_1|\nabla |F_1|}{|F_1|^2} \right\|_{L^p(\Omega,|F_1|^p)} 
				    + \left\| \frac{|F_1|\nabla |F_1|- |F_2|\nabla |F_1|}{|F_1|^2} \right\|_{L^p(\Omega,|F_1|^p)}\\
				    &=\left\| \nabla|F_1| - \nabla|F_2|\right\|_{L^p(\Omega)} + \left\| \frac{\nabla |F_1|}{|F_1|} \left(|F_1|-|F_2|\right) \right\|_{L^p(\Omega)}.
 \end{aligned}
\end{equation}
By combining the Estimates \eqref{est:unconstrainedstep1},\eqref{est:poincare} and \eqref{est:gradientmod} we arrive at the desired bound in \eqref{est:firststab}.
\end{proof}
\subsection{Cheeger's inequality}
Next we want to provide some insight on the Poincar\'e constant $C_P(\Omega,|F_1|^p,p)$, which by Theorem \ref{thm:firststab} is closely related to the question of local stability at $F_1$.\\
The following result is inspired by the work of Jeff Cheeger in \cite{cheeger}, where the smallest eigenvalue of the Laplacian on a Riemannian manifold is related to a geometric 
quantity which is similar to the Cheeger constant as introduced in Definition \ref{def:cheegerconst}.
\begin{theorem}\label{thm:cheeger}
 Let $1\le p \le 2$, let $\Omega\subset\R^{2d}$ be a domain and let $w$ be a nonnegative and continuous weight on $\Omega$.
 Then it holds that
 $$
 C_P(\Omega,w,p)\le 8 \cdot h(w,\Omega)^{-1}.
 $$
\end{theorem}
\begin{proof}
 A proof for the case $d=1$ is carried out in the Appendix of our preceding paper \cite{stablegaborpr} and generalizes readily to the multivariate case.
\end{proof}
Theorem \ref{thm:cheeger} immediately implies the following version of the stability result Theorem \ref{thm:firststab}.
\begin{corollary}\label{cor:stabcheeger}
 Let $\Omega\subset \Cd$ be a domain and let $1\le p \le 2$.
 Suppose that $F_1,F_2\in L^p(\Omega)$ are such that their quotient $F_2/F_1$ is meromorphic and suppose that $|F_1|$ is continuous.
Then it holds that 
 \begin{multline*}
  \inf_{|c|=1} \|F_2-cF_1\|_{L^p(\Omega)} \le   \||F_2|-|F_1|\|_{L^p(\Omega)} \\
  +2^{9/2} \cdot h(|F_1|^p,\Omega)^{-1} \left( \|\nabla |F_1|-\nabla |F_2|\|_{L^p(\Omega)} 
  + \left\| \frac{\nabla |F_1|}{|F_1|} \left(|F_1|-|F_2|\right) \right\|_{L^p(\Omega)} \right).
 \end{multline*}
\end{corollary}
\section{On the growth of the logarithmic derivative of holomorphic functions}\label{sec:logderiv}
The estimates in Theorem \ref{thm:firststab} and Corollary \ref{cor:stabcheeger} include the logarithmic derivative of the modulus of $F_1$, 
a term which is rather undesirable as we want to obtain a bound which depends on the difference of $|F_1|$ and $|F_2|$ only.\\
This section is devoted to the study of the logarithmic derivatives of entire functions that satisfy certain growth restrictions.
More precisely, we consider the following class of entire functions on $\C^d$.
  \begin{definition}\label{def:Oab}
   Let $\alpha, \beta >0$, then 
   $$
   \Oab(\C^d) := \{ G\in \mathcal{O}(\C^d): ~ M_G(r)\le \abs{G(0)} e^{\alpha r^\beta} ~\forall r>0\},
   $$
   where we set $M_G(r):=\max_{\abs{z}\le r} \abs{G(z)}$.
  \end{definition}
\begin{remark}
Note that we require a pointwise inequality to hold, whereas in the definition of type and order a similar inequality only needs to hold in an asymptotic sense.
Consequently for an entire function $G$ of type $\tau\le a$ and order $\sigma\le b$ in general
$G\notin \Oab(\C^d).$
\end{remark}
The quantity of our interest is the $L^p$-norm of $(\log G)'=G'/G$ on balls centered at the origin.
Our results reveal that $\norm{(\log G)'}_{L^p(B_r)}$ grows at most polynomially in $r$ and provide
explicit bounds that depend on $\alpha, \beta$ but are remarkably independent of $G\in\Oab(\C^d)$.\\
The main theorem of this section reads as follows.
\begin{theorem}\label{thm:main}
 Let $1\le p <1+1/(2d-1)$. There exists a constant $c>0$ that only depends on $d$ and $p$ such that for all $G\in \Oab(\C^d)$, $G\neq 0$ and all $r>0$ the estimate 
 \begin{equation}
 \norm{(\log G)'}_{L^p(B_r)} \le c  \alpha 2^{2d+2\beta}r^{2d+\beta-1}
 \end{equation}
 holds.
\end{theorem}
The results of this section rely heavily on the formula of \emph{Poisson-Jensen}.
In the onedimensional case the formula is well-known, see \cite{ahlfors}.
In higher dimensions a similar formula is established for subharmonic functions in potential theory, cf. \cite{hayman}.
Since $\log \abs{G}$ is subharmonic for any holomorphic function $G$ \cite{ronkin}, the formula can be applied on $\log \abs{G}$ and leads to the following result:
\begin{theorem}[Poisson-Jensen]\label{thm:poissonjensen1}
Suppose $G:\C^d\rightarrow \C$ is entire.\\
In case $d=1$ let $z_1,z_2,\ldots$ denote the zeros of $G$ repeated according to multiplicity. 
 Then for $r>0$ and $\abs{z}<r$ it holds that 
 \begin{equation}\label{eq:poissonjensen1}
	       \log \abs{G(z)} =\frac{1}{2\pi} \int_0^{2\pi} \log \abs{G(r e^{i\theta})} \frac{r^2-|z|^2}{|re^{i\theta}-z|^2}~d\theta - 
	       \sum_{k:\abs{z_k}<r} \log \abs{\frac{r^2-\overline{z_k}z}{r(z-z_k)}}.
 \end{equation}
In case $d\ge2$ there exists a Borel measure $\mu_G$ on $\C^d$ such that for any $r>0$ and $\abs{z}<r$
 \begin{equation}\label{eq:poissonjensen2}
 \begin{split}
  \log \abs{G(z)}= &~ \frac{1}{S_{d-1}\cdot r} \int_{\partial B_r} \log\abs{G(\xi)} \frac{r^2-\abs{z}^2}{\abs{z-\xi}^{2d}} ~d\sigma(\xi)\\
  &- \int_{B_r} \frac{1}{\abs{z-\xi}^{2d-2}} - \left(\frac{r}{\abs{\xi}\abs{z-\xi r^2 / \abs{\xi}^2}}\right)^{2d-2} ~d\mu_G(\xi)
 \end{split}
 \end{equation}
 holds true, where $S_{d-1}$ denotes the surface area of the unit sphere and $\sigma$ denotes the surface measure on $\partial B_r$.
\end{theorem}
Since the formula of Poisson-Jensen takes different shapes depending on the dimension, in the following we will consider the cases $d=1$ and $d\ge 2$ seperately.
Note however that qualitatively Equations \eqref{eq:poissonjensen1} and \eqref{eq:poissonjensen2} are quite similar:
First of all, both the integral in \eqref{eq:poissonjensen1} and the first integral in \eqref{eq:poissonjensen2} express a weighted average of $\log |G|$ over the surface of a ball.
Secondly, the sum in \eqref{eq:poissonjensen1} can be rewritten as
$$
\sum_{k:\abs{z_k}<r} \log \abs{\frac{r^2-\overline{z_k}z}{r(z-z_k)}} = \int_{B_r} \log \abs{\frac{r^2-\overline{\xi}z}{r(z-\xi)}} ~d\mu(\xi),
$$
where $\mu:=\sum_k \delta_{z_k}$, and therefore is the integral over a function with singularity in $z$ w.r.t. a measure that is supported precisely on the zero set of $G$.
The second integral in \eqref{eq:poissonjensen2} can be interpreted similarly.
More generally, the measure $\mu_G$ is related to the distribution of the zeros of $G$.
As we will see next the distribution of zeros can be controlled in terms of the growth of $G$.
\begin{proposition}\label{lem:zerosandgrowth}
 Let $\alpha,\beta >0$ and suppose $G\in \Oab(\C^d)$ be such that $G\neq 0$.\\ 
 In case $d=1$ let $z_1,z_2,\ldots$ denote the zeros of $G$ repeated according to multiplicity, then for all $r>0$
		   $$\sharp\{k: \abs{z_k}<r\} \le \frac{2^\beta \alpha}{\log2} r^\beta.$$ 
 In case $d\ge2$ let $\mu_G$ be defined by \eqref{eq:poissonjensen2} and $\nu_G$ by $\nu_G(r):=\int_{B_r} \abs{z}^{-2d+2}~d\mu_G(z)$.
		    Then for all $r>0$
		     $$ \mu_G(B_r) \le \alpha 2^{\beta+1} r^{2d-2+\beta} \quad \text{and} \quad \nu_G(r)\le \alpha 2^{\beta+1} r^\beta.$$
\end{proposition}
\begin{proof}
 Note that $G\neq 0$ together with $G\in \Oab(\C^d)$ implies that $G(0)\neq 0$.
 Since both Equation \eqref{eq:poissonjensen1} and \eqref{eq:poissonjensen2} are invariant w.r.t. multiplication of $G$ by a non zero constant, we may assume w.l.o.g. that $|G(0)|=1$.\\

 We first prove the statement for $\mathbf{d=1}$.
 Since $\abs{G(0)}=1$ applying \eqref{eq:poissonjensen1} for $z=0$ yields 
 \begin{equation}\label{eq:jensen}
 \sum_{k:\abs{z_k}<r} \log \frac{r}{\abs{z_k}}=\frac{1}{2\pi} \int_0^{2\pi} \log \abs{G(r e^{i\theta})}~d\theta.
 \end{equation}
 Since for $\abs{z_k}<r/2$ it holds that $\log \frac{r}{\abs{z_k}}\ge \log2$ we can estimate
 \begin{eqnarray*}
  \sharp \{k: \abs{z_k}<r/2\} &\le & \frac1{\log2} \sum_{k:\abs{z_k}<r} \log \frac{r}{\abs{z_k}}\\
			     &= & \frac1{\log2} \frac{1}{2\pi} \int_0^{2\pi} \log \abs{G(r e^{i\theta})}~d\theta\\
			     &\le & \frac{\alpha}{\log2} r^\beta. 
 \end{eqnarray*}
 Substitution of $r$ by $2r$ concludes the proof of the first statement.\\
 
 We proceed with the case $\mathbf{d\ge2}$.
 To simplify notation we set $\mu:=\mu_G$ and $\nu:=\nu_G$.
 Applying Equation \eqref{eq:poissonjensen2} for $z=0$ gives 
 \begin{equation}\label{eq:jensenrearranged}
 \int_{B_r} \frac{1}{\abs{\xi}^{2d-2}} -\frac{1}{r^{2d-2}}~d\mu(\xi) = \frac{1}{S_{d-1} \cdot r^{2d-1}} \int_{\partial B_r} \log \abs{G(\xi)}~d\sigma(\xi).
\end{equation}
Since for $\abs{\xi}<r/2$ the inequality 
$$
\frac{1}{\abs{\xi}^{2d-2}} \le 2\cdot \left(\frac{1}{\abs{\xi}^{2d-2}}-\frac{1}{r^{2d-2}}\right)
$$
holds, we estimate by using \eqref{eq:jensenrearranged} that
\begin{eqnarray*}
 \nu(r/2) &=& \int_{B_{r/2}} \abs{\xi}^{-2d+2} ~d\mu(\xi)\\
	  &\le& \frac2{S_{d-1}\cdot r^{2d-1}} \int_{\partial B_r} \log \abs{G(\xi)}~d\sigma(\xi)\\
	  &\le& 2 \alpha r^\beta.
\end{eqnarray*}
Substition of $r$ by $2r$ yields that $$\nu(r)\le 2^{\beta+1}\alpha r^{\beta}.$$
Since 
$
\nu(r)\ge r^{-2d+2}\mu(B_r)
$
the second claim follows immediately, i.e.,
$$
\mu(B_r) \le 2^{\beta+1} \alpha r^{2d-2+\beta}.
$$
\end{proof}
Before we go on to prove Theorem \ref{thm:main} let us provide some intuition for the case $d=1$.
In that case the zero set of any entire function $G\neq 0$ is discrete.
Locally at a zero $z_0$ we can factorize
$$
G(z)=(z-z_0)^m\cdot H(z),
$$
where $m\in \N$ denotes the multiplicity of the zero $z_0$ and $H$ is a locally nonvanishing, analytic function.
Computing the logarithmic derivative of $G$ gives
$$
(\log G)'(z)=\frac{G'(z)}{G(z)} =\frac{m}{z-z_0} + (\log H)'(z),
$$
thus the logarithmic derivative has a pole of order $1$ at $z_0$.
Proposition \ref{lem:zerosandgrowth} allows us to control the number of zeros.
If we choose $p$ such that $z\mapsto \abs{z}^{-1}$ is $L^p$-integrable we will be able to bound $\norm{(\log G)'}_{L^p(B_r)}$.\\
For $d\ge2$ the situation is more complicated as zeros are not discrete any more.
Before we prove Theorem \ref{thm:main} we derive pointwise estimates, again by exploiting the representation formula of Poisson-Jensen.
\begin{proposition}\label{prop:ptwest}
 Let $\alpha, \beta>0$. 
 Suppose $G\in \Oab(\C^d)$ is such that $G\neq 0$.
 Then there exists a constant $c$ that only depends on $d$ such that for all $r>0$ and $\abs{z}<r/2$
 \begin{equation}\label{est:ptw}
  \abs{(\log G)'(z)} \le c \left( \alpha 2^\beta r^{\beta-1} + \int_{B_r} \frac{d\mu(\xi)}{\abs{z-\xi}^{2d-1}}\right),
 \end{equation}
 where $\mu:=\sum_k \delta_{z_k}$ -- where $z_k$ are the zeros of $G$ repeated according to multiplicity -- in case $d=1$ and 
 $\mu=\mu_G$ is defined by \eqref{eq:poissonjensen2} in case $d\ge 2$.
\end{proposition}
\begin{proof}
 The assumption that $G\in \Oab(\C^d)$ is not the zero function implies that $G(0)\neq 0$. Since the logarithmic derivative is invariant w.r.t. multiplication by a nonzero constant we may assume that 
 $G(0)=1$.
 Again the cases $d=1$ and $d\ge 2$ are treated seperately:\\

 \bm{$d=1:$}
 First we compute $\log G(z)$ for $\abs{z}<r$ using \eqref{eq:poissonjensen1}.
 \begin{equation}
 \begin{split}
  \log G(z) &= 2 \log \abs{G(z)} -\log \overline{G(z)}\\
	    &= \frac1\pi \int_0^{2\pi} \log \abs{G(re^{i\theta})} \cdot 
	    \frac12\left( 
	    \frac{re^{i\theta}+z}{re^{i\theta}-z} + \frac{re^{-i\theta}+\overline{z}}{re^{-i\theta}-\overline{z}} 
	    \right)d\theta\\
	    &~-\sum_{k:\abs{z_k}<r} \left(\log \frac{r^2-\overline{z_k}z}{r(z-z_k)} + \log \frac{r^2-z_k\overline{z}}{r(\overline{z}-\overline{z_k})}\right) - \log \overline{G(z)}
 \end{split}
 \end{equation}
Next we differentiate w.r.t. $z$. The antiholomorphic terms are annihilated by $\frac\partial{\partial z}$, i.e.,
$$
\frac\partial{\partial z} \left(\frac{re^{-i\theta}+\overline{z}}{re^{-i\theta}-\overline{z}}\right)=0,
\quad \frac\partial{\partial z} \left(\log \frac{r^2-z_k\overline{z}}{r(\overline{z}-\overline{z_k})}\right)= 0
\quad \text{and} \quad \frac\partial{\partial z}\log \overline{G(z)}=0.
$$
Elementary computations show that 
$$
\frac\partial{\partial z} \left(\frac12 \frac{re^{i\theta}+z}{re^{i\theta}-z}\right) = \frac{re^{i\theta}}{(re^{i\theta}-z)^2} \quad 
\text{and} \quad \frac\partial{\partial z} \left(\log \frac{r^2-\overline{z_k}z}{r(z-z_k)}\right)=\frac{\abs{z_k}^2-r^2}{(z-z_k)(r^2-\overline{z_k}z)}
$$
and thus
\begin{equation}\label{eq:logderiv11}
 \left(\log G\right)'(z) = \frac1\pi \int_0^{2\pi} \log \abs{G(re^{i\theta})} \frac{re^{i\theta}}{(re^{i\theta}-z)^2} d\theta + \sum_{k:\abs{z_k}<r} \frac{r^2-\abs{z_k}^2}{(z-z_k)(r^2-\overline{z_k}z)} =:I(z) + II(z)
\end{equation}
We will now estimate $\abs{(\log G)'(z)}$ for $\abs{z}<r/2$.
We treat $I$ and $II$ seperately.
\begin{itemize}
 \item[Estimating $\abs{I(z)}$:] Note that applying Cauchy's integral formula on the function $z\mapsto z$ yields
			         \begin{equation*}\label{eq:vanishingint}
			           \int_0^{2\pi} \frac{re^{i\theta}}{(z-re^{i\theta})^2} ~d\theta =0 \quad \text{for all }~z \in B_r.
			         \end{equation*}
				 Therefore we have
				 $$
				 I(z)=\frac1\pi \int_0^{2\pi} (\log \abs{G(re^{i\theta})}-\log \abs{M_G(r)}) \frac{re^{i\theta}}{(z-re^{i\theta})^2} ~d\theta
			         $$
			         and furthermore for $\abs{z}<r/2$
			         \begin{eqnarray*}
			          \abs{I(z)} &\le & \frac1\pi \int_0^{2\pi} \abs{\log \abs{G(re^{i\theta})}-\log M_G(r)} \cdot \abs{\frac{re^{i\theta}}{(z-re^{i\theta})^2}} d\theta\\
			                     &\le & \frac4{\pi r} \int_0^{2\pi} \log M_G(r)-\log \abs{G(re^{i\theta})} d\theta
			         \end{eqnarray*}
				 due to $\abs{G(re^{i\theta})} \le M_G(r)$.
				 By having a look at Equation \eqref{eq:jensen} we observe that $\int_0^{2\pi} \log \abs{G(re^{i\theta})} d\theta$ is nonnegative and therefore
				 \begin{equation}\label{est:I1final}
				  \abs{I(z)} \le \frac{4\alpha}{\pi} r^{\beta-1} \quad \text{for } \abs{z}<r/2.
				 \end{equation}

 \item[Estimating $\abs{II(z)}$:] For any $\abs{z}<r/2$ and $\abs{z_k}<r$ we have 
				  $$
				  \abs{r^2-\abs{z_k}^2} \le r^2 \quad\text{and} \quad \abs{r^2-\overline{z_k}z}\ge r^2-\abs{z_k}\abs{z}\ge r^2/2;
				  $$
				  making use of these estimates yields
				  \begin{equation}\label{est:I2final}
				   \abs{II(z)} \le \sum_{k:\abs{z_k}<r} \abs{\frac{r^2-\abs{z_k}^2}{(z-z_k)(r^2-\overline{z_k}z)}}\\
					       \le 2 \sum_{k:\abs{z_k}<r} \abs{z-z_k}^{-1}
				  \end{equation}		  
\end{itemize}
Combining \eqref{est:I1final} and \eqref{est:I2final} implies \eqref{est:ptw} for $d=1$.\\

\bm{$d\ge2:$}
Let $\nu:=\nu_G$ be defined as in Proposition \ref{lem:zerosandgrowth}. Then by utilizing Equation \eqref{eq:poissonjensen2} we know that
\begin{equation}\label{eq:logG}
 \log G(z)=-\log \overline{G(z)} + \frac2{S_{d-1} r}\int_{\partial B_r} \log \abs{G(\xi)} \cdot h(z,\xi)~d\sigma(\xi)-2\int_{B_r} k(z,\xi)~d\mu(\xi)
\end{equation}
for all $\abs{z}<r$, where
$$
h(z,\xi)=\frac{r^2-\abs{z}^2}{\abs{z-\xi}^{2d}}
$$
and
$$
k(z,\xi)=\frac1{\abs{z-\xi}^{2d-2}} - \left(
\frac{r}{\abs{\xi}\cdot \abs{z-\frac{\xi r^2}{\abs{\xi}^2}}}
\right)^{2d-2}
$$
We differentiate Equation \eqref{eq:logG} w.r.t. the first component $z_1$ of $z$ (differentation w.r.t. the other variables works in the exact same way).
Interchanging order of integration and differentiation yields -- since $\log \overline{G}$ is antiholomorphic w.r.t. $z_1$ -- that
\begin{equation}\label{eq:logderivG}
\begin{split}
 \frac{\partial}{\partial z_1}\log G(z) &= \frac2{S_{d-1} r}\int_{\partial B_r} \log \abs{G(\xi)} 
 \cdot \frac{\partial}{\partial z_1} h(z,\xi)~d\sigma(\xi)-2\int_{B_r} \frac{\partial}{\partial z_1} k(z,\xi)~d\mu(\xi)\\
					&=: III(z)+IV(z).
 \end{split}
 \end{equation}
To compute the derivative of the kernel function $h$ we write
$$
h(z,\xi)=\frac{r^2-z_1\overline{z_1}-\abs{z'}^2}{\left((z_1-\xi_1)(\overline{z_1-\xi_1})+\abs{z'-\xi'}^2\right)^d},
$$
where $z=(z_1,z')$ and $z'\in \C^{d-1}$ and similarly for $\xi$.
Then 
\begin{equation}\label{eq:derivh}
 \begin{split}
 \frac{\partial}{\partial z_1} h(z,\xi) &= \frac{-\overline{z_1}\abs{z-\xi}^{2d}-(r^2-\abs{z}^2)d\abs{z-\xi}^{2(d-1)}(\overline{z_1-\xi_1})}{\abs{z-\xi}^{4d}}\\
					 &= \abs{z-\xi}^{-2d-2} \left(-\overline{z_1}\abs{z-\xi}^2-d(r^2-\abs{z}^2)(\overline{z_1-\xi_1})\right)
\end{split}
\end{equation}
where we used that $z_1\mapsto \overline{z_1}$ is antiholomorphic.
A similar computation yields
\begin{equation}\label{eq:derivk}
 \frac{\partial}{\partial z_1} k(z,\xi) = (d-1)\cdot \left( \frac{r^{2d-2}}{\abs{\xi}^{2d-2}}\cdot 
 \frac{\overline{z_1-\hat{\xi}_1}}{\abs{z-\hat{\xi}}^{2d}}-\frac{\overline{z_1-\xi_1}}{\abs{z-\xi}^{2d}}
 \right)
\end{equation}
where we set $\hat{\xi}=\xi r^2/\abs{\xi}^2$.
Next we derive bounds for $\abs{III(z)}$ and $\abs{IV(z)}$ for $\abs{z}<r/2$.
\begin{itemize}
 \item[Estimating $\abs{III(z)}$:] 
 For $\abs{\xi}=r$ we can now estimate
 \begin{equation*}
   \abs{\frac{\partial}{\partial z_1} h(z,\xi)} 
						\lesssim_d r^{-2d+1},
 \end{equation*}
 where the symbol $"\lesssim_d"$ means that the left hand side can be bounded by the right hand side times a constant that depends on $d$ only.
 Thus 
\begin{eqnarray*}
\abs{III(z)} &\le &  \frac2{S_{d-1} \cdot r}\int_{\partial B_r} \abs{\log \abs{G(\xi)} } \cdot \abs{\frac{\partial}{\partial z_1} h(z,\xi)}~d\sigma(\xi)\\
	     &\lesssim_d& r^{-2d} \int_{\partial B_r} \abs{\log \abs{G(\xi)}} ~d\sigma(\xi)\\
	     &\le& r^{-2d} \left( \int_{\partial B_r} \abs{\log M_G(r)-\log \abs{G(\xi)}}  ~d\sigma(\xi) + \int_{\partial B_r} \abs{\log M_G(r)}  ~d\sigma(\xi) \right)\\
	     &=& r^{-2d} \left( \int_{\partial B_r} \log M_G(r)-\log \abs{G(\xi)}  ~d\sigma(\xi) + \int_{\partial B_r} \abs{\log M_G(r)}  ~d\sigma(\xi) \right)
\end{eqnarray*}
From \eqref{eq:jensenrearranged} it follows that $\int_{\partial B_r} \log \abs{G(\xi)} d\sigma(\xi)$ is nonnegative.
Since $G$ is holomorphic and $G(0)=1$ we have $\abs{\log M_G(r)} = \log M_G(r)$ for all $r$.
Therefore we can further estimate
\begin{equation*}
 \abs{III(z)} \lesssim_d  r^{-2d} \int_{\partial B_r} \log M_G(r) ~d\sigma(\xi)
	       \lesssim_d  r^{-2d} r^{2d-1} \alpha r^{\beta} 
\end{equation*}
and thus 
\begin{equation}\label{est:IIIptw}
 \abs{III(z)} \lesssim_d \alpha r^{\beta-1}\quad \text{for } \abs{z}<r/2.
\end{equation}
 \item[Estimating $\abs{IV(z)}$:]
 First note that $\abs{\hat{\xi}}>r$ whenever $\abs{\xi}<r$. For $\abs{z}<r/2$ we estimate that
 $$
 \abs{\frac{\partial}{\partial z_1} k(z,\xi)} \lesssim_d \frac{r^{2d-2}}{\abs{\xi}^{2d-2}}\cdot r^{-2d+1}+\frac1{\abs{z-\xi}^{2d-1}}
 $$
 Therefore 
 \begin{eqnarray*}
  \abs{IV(z)} &\le& 2\int_{B_r} \abs{ \frac{\partial}{\partial z_1} k(z,\xi)}~d\mu(\xi)\\
              &\lesssim_d& r^{-1} \int_{B_r} \abs{\xi}^{-2d+2}~d\mu(\xi)+\int_{B_r} \frac{d\mu(\xi)}{\abs{z-\xi}^{2d-1}}\\
              &=&r^{-1}\cdot \nu(r) + \int_{B_r} \frac{d\mu(\xi)}{\abs{z-\xi}^{2d-1}},
 \end{eqnarray*}
 where we set $\nu(r):=\int_{B_r} \abs{\xi}^{-2d+2}d\mu(\xi)$.
By Proposition \ref{lem:zerosandgrowth} it holds that
\begin{equation}\label{est:IVptw} 
 \abs{IV(z)} \lesssim_d \alpha 2^\beta r^{\beta-1} + \int_{B_r} \frac{d\mu(\xi)}{\abs{z-\xi}^{2d-1}}\quad \text{for } \abs{z}<r/2.
\end{equation}
\end{itemize}
Combining \eqref{est:IIIptw} and \eqref{est:IVptw} yields \eqref{est:ptw} for $d\ge 2$.
\end{proof}
We are set to prove Theorem \ref{thm:main}.
\begin{proof}[Proof of Theorem \ref{thm:main}]
 Let $\mu$ be defined as in Proposition \ref{prop:ptwest}, then the pointwise estimate
 \begin{equation}\label{logderivG:ptwagain}
   \abs{(\log G)'(z)} \le c \left( \alpha 2^\beta r^{\beta-1} + \int_{B_r} \frac{d\mu(\xi)}{\abs{z-\xi}^{2d-1}}\right)
 \end{equation}
 holds for $z\in B_{r/2}$, where $c$ only depends on $d$.
 We begin by bounding the norm of the second term on the right hand side:
 W.l.o.g. we may assume that $\mu(B_r)>0$, otherwise there is nothing to estimate.
 \begin{eqnarray*}
  \norm{z\mapsto  \int_{B_r} \frac{d\mu(\xi)}{\abs{z-\xi}^{2d-1}}}_{L^p(B_{r/2})}^p &=& \int_{B_{r/2}} \left(  \int_{B_r} \frac{d\mu(\xi)}{\abs{z-\xi}^{2d-1}} \right)^p dA(z)\\
											 &=& \int_{B_{r/2}} \left(  \int_{B_r} \frac{\mu(B_r)}{\abs{z-\xi}^{2d-1}} \cdot \frac{d\mu(\xi)}{\mu(B_r)} \right)^p dA(z)\\
											 &\le& \mu(B_r)^{p-1} \int_{B_{r/2}} \int_{B_r}\frac{d\mu(\xi)}{\abs{z-\xi}^{(2d-1)p}} dA(z)
 \end{eqnarray*}
where we used Jensen's inequality. By interchanging order of integration we obtain
$$
 \norm{z\mapsto  \int_{B_r} \frac{d\mu(\xi)}{\abs{z-\xi}^{2d-1}}}_{L^p(B_{r/2})}^p \le \mu(B_r)^{p-1} \int_{B_r} \int_{B_{r/2}}\frac{dA(z)}{\abs{z-\xi}^{(2d-1)p}} d\mu(\xi).
$$
The inner integral can be bounded by
$$
\int_{B_{r/2}}\frac{dA(z)}{\abs{z-\xi}^{(2d-1)p}} \le \int_{B_{r/2}}\frac{dA(z)}{\abs{z}^{(2d-1)p}} = \norm{z\mapsto \abs{z}^{-2d+1}}_{L_p(B_{r/2})}^p=:c_{d,p}^p.
$$
Note that $p<1+1/(2d-1)$ is precisely the condition for $c_{d,p}$ to be finite.
Thus we arrive at
\begin{equation}\label{est:2ndsummand}
\norm{z\mapsto  \int_{B_r} \frac{d\mu(\xi)}{\abs{z-\xi}^{2d-1}}}_{L^s(B_{r/2})}^p \le c_{d,p}^p \mu(B_r)^p.
\end{equation}
To estimate the first summand of the right hand side in \eqref{logderivG:ptwagain} we compute the norm of the constant function
\begin{equation}\label{est:1stsummand}
 \| 1\|_{L^p(B_{r/2})} = vol(B_{r/2})^{1/p} = \left( \frac{\pi^d}{d!} (r/2)^{2d}\right)^{1/p}.
\end{equation}
It follows from \eqref{logderivG:ptwagain}, \eqref{est:1stsummand}, \eqref{est:2ndsummand} and Proposition \ref{lem:zerosandgrowth} that there exists a $c'$ that only depends on $d$ and $p$ such that 
\begin{eqnarray*}
 \norm{(\log G)'}_{L^p(B_{r/2})} &\le & c' \left( \alpha2^\beta r^{\beta-1} \|1\|_{L^p(B_{r/2})} + \mu(B_r) \right)\\
				    & \le & c'' \left( \alpha2^\beta r^{\beta-1} \cdot (r/2)^{2d/p} +\alpha 2^{\beta+1}r^{2d-2+\beta}\right)\\
				    &\le & c''' \alpha 2^{\beta+1}r^{2d+\beta-1},
\end{eqnarray*}
where $c''$ and $c'''$ again only depend on $d$ and $p$.
By subtituting $r$ by $2r$ we get the desired bound
$$
\norm{(\log G)'}_{L^p(B_r)} \le c''' \alpha 2^{2d+2\beta} r^{2d+\beta-1}.
$$
\end{proof}

\section{Stable Gabor phase retrieval}
\subsection{The main result}
In the present section we will elaborate on how the results from the previous two sections enable us to derive stability estimates for the problem of phase retrieval from Gabor magnitudes.
The Gabor transform $\mathcal{G}f$ possesses the pleasant property that it is an entire function (up to an exponential factor and a reflection) and therefore the tools we developed thus far can be applied.
\begin{lemma}\label{lem:gaborentire}
 Let $\eta(z)=e^{\pi|z|^2/2-\pi i x\cdot y}$.
 Then we have that for any $f\in \mathcal{S}'(\R^d)$ the function $z\mapsto \mathcal{G}f(\bar{z})\eta(z)$ is entire.
\end{lemma}
\begin{proof}
 A proof for functions of polynomial growth can be found in \cite{grochenig}.
\end{proof}
Due to Lemma \ref{lem:gaborentire} we can apply Corollary \ref{cor:stabcheeger} to $F_1(z)=\mathcal{G}f(\bar{z})$ and $F_2(z)=\mathcal{G}g(\bar{z})$ to obtain the following result.
\begin{corollary}\label{cor:firststabgabor}
 Let $\Omega\subset \R^{2d}$ be a domain. Then for all $f,g\in \mathcal{M}^p(\R^d)$ it holds that
 \begin{multline*}
  \inf_{|c|=1} \|\mathcal{G}g-c\mathcal{G}f\|_{L^p(\Omega)} \le \||\mathcal{G}g|-|\mathcal{G}f|\|_{L^p(\Omega)}\\ + 2^{9/2}\cdot h(|\mathcal{G}f|^p,\Omega)^{-1} \left( \|\nabla |\mathcal{G}f|-\nabla |\mathcal{G}g|\|_{L^p(\Omega)} 
  + \left\| \frac{\nabla |\mathcal{G}f|}{|\mathcal{G}f|} \left(|\mathcal{G}f|-|\mathcal{G}g|\right) \right\|_{L^p(\Omega)} \right).
 \end{multline*}
\end{corollary}
Except for the logarithmic derivative Corollary \ref{cor:firststabgabor} already gives an estimate of the desired form.
The following proposition aims at absorbing the logarithmic derivative in a polynomial weight which does not depend on $f$.
\begin{proposition}\label{prop:absorblogderiv}
 Let $\Omega\subset \R^{2d}$ and let $1\le p <2d/(2d-1)$ and $q>p/(1-p\frac{2d-1}{2d})$.
 Suppose that $f\in \mathcal{M}^{\infty}(\R^d)$ such that $|\mathcal{G}f|$ has a global maximum at $z_0$.
 Then there exists a constant $c$ which only depends on $d,p$ and $q$ such that for all measurable functions $H$ it holds that
 $$
 \left\| \frac{\nabla |\mathcal{G}f|}{|\mathcal{G}f|} H \right\|_{L^p(\Omega)} \le c \left\| (1+|\cdot-z_0|)^{2d+2} H\right\|_{L^q(\Omega)}.
 $$
\end{proposition}
\begin{proof}
 We assume w.l.o.g. that $\Omega =\R^{2d}$ and that $z_0=0$(otherwise translate and modulate $f$).
 The proof is split up into two parts: 
 First we derive uniform bounds of the norm of the logarithmic derivative on balls centered at the origin.
 In the second part the logarithmic derivative is absorbed in a polynomial weight by making use of Hölder's inequality and Part 1.
 \\
 {\bf Part 1:}
 By Lemma \ref{lem:gaborentire} we know that 
 $G(z):=\mathcal{G}f(\bar{z}) \eta(z)$
 is entire.
 The gradient of the modulus of the Gabor transform can be computed in terms of $G$:
 \begin{equation*}
 \nabla |\mathcal{G}f| = \nabla \left( |G| e^{-\frac{\pi}2 |\cdot|^2}\right) = (\nabla|G|) e^{-\frac{\pi}2|\cdot|^2} + |G|(\nabla e^{-\frac{\pi}2|\cdot|^2})
 \end{equation*}
 Since $|\nabla e^{-\frac{\pi}2|\cdot|^2}|(z)=\pi |z|$ we obtain 
 \begin{equation}
   |\nabla \log|\mathcal{G}f|(z)| \le |\nabla \log|G| (z)|+\pi |z|.
 \end{equation}
Lemma \ref{lem:keylemma} implies that the right hand side coincides with $2^{-1/2}|(\log G)'(z)| + \pi |z|$ almost everywhere.\\
 The assumption that $|\mathcal{G}f|$ has a maximum at the origin implies that $G\in \mathcal{O}_{\pi/2}^2(\C^d)$, see Definition \ref{def:Oab}.
 We can therefore apply Theorem \ref{thm:main} to obtain for $r>0$ and $1\le s<1+ 1/(2d-1)$ that 
 \begin{equation}\label{est:logderivspectrogram}
   \|\nabla \log |\mathcal{G}f|\|_{L^s(B_r)} \le 2^{-1/2} \|(\log G)'\|_{L^s(B_r)} + \pi\|z\mapsto z\|_{L^s(B_r)} \lesssim r^{2d+1},
 \end{equation}
where the implicit constant depends on $d$ and $s$ only.\\
{\bf Part 2:}
We define $s$ by the equation $1/p=1/q+1/s$. One can elementary verify that the assumptions on $p$ and $q$ imply that $1\le s < 1+1/(2d-1)$.
Thus the $L^s$-norm of the logarithmic derivative can be bounded as in Part 1, see Equation \eqref{est:logderivspectrogram}.\\
Let $D_0:=B_1$ and $D_j:=B_{2^j}\setminus B_{2^{j-1}}$ for $j\in \mathbb{N}$.
Then we have that
$$
\|\nabla \log|\mathcal{G}f|\cdot H\|_{L^p(\R^{2d})}^p = \sum_{j\ge 0} \|\nabla \log|\mathcal{G}f|\cdot H\|_{L^p(D_j)}^p.
$$
We apply now Hölder's inequality on every $D_j$ to obtain
\begin{equation*}
  \|\nabla \log|\mathcal{G}f| \cdot H\|_{L^p(D_j)} \le \|\nabla \log |\mathcal{G}f|\|_{L^s(D_j)} \|H\|_{L^q(D_j)}
						  \lesssim 2^{j(2d+1)} \|H\|_{L^q(D_j)},
\end{equation*}
where we used Estimate \eqref{est:logderivspectrogram} from Part 1.
Let $r:=q/p>1$ and $r'$ its Hölder conjugate, i.e., $1/r+1/r'=1$.
By applying Hölder's inequality for sums we estimate further
\begin{equation*}
 \begin{aligned}
  \|\nabla \log|\mathcal{G}f| \cdot H\|_{L^p(\R^{2d})}^p &\lesssim \sum_{j \ge 0} 2^{j(2d+1)p} \|H\|_{L^q(D_j)}^p\\
							&= \sum_{j\ge 0} 2^{-j/r'} \cdot 2^{j(2dp+p+1/r')} \|H\|_{L^q(D_j)}^p\\
							&\le \left( \sum_{j\ge 0} 2^{-j}\right)^{1/r'} \cdot \left(\sum_{j\ge 0} 2^{j(2dp+p+1/r')r} \|H\|_{L^q(D_j)}^q \right)^{p/q}.
 \end{aligned}
\end{equation*}
The first factor is a finite constant depending on $r'$ and therefore ultimately on $p$ and $q$ only.\\
The second factor is estimated in the following way:
\begin{equation*}
 \begin{aligned}
  \sum_{j\ge 0} 2^{j(2dp+p+1/r')r} \|H\|_{L^q(D_j)}^q &\le \sum_{j\ge 0} \int_{D_j} 2^{j(2d+2)pr} |H(z)|^q ~dA(z)\\
						      &=  \sum_{j\ge 0} \int_{D_j} 2^{j(2d+2)q} |H(z)|^q ~dA(z)\\
						      &\lesssim \|(1+|\cdot|^{2d+2}) H\|_{L^q(\R^{2d})}^q,
 \end{aligned}
\end{equation*}
where we used that for $z\in D_j$ it holds that $2^{j(2d+2)}\lesssim 1+|z|^{2d+2}$.
Thus we get the desired estimate
\begin{equation*}
 \|\nabla \log|\mathcal{G}f| \cdot H\|_{L^p(\R^{2d})} \lesssim \|(1+|\cdot|^{2d+2}) H\|_{L^q(\R^{2d})}.
\end{equation*}
\end{proof}
The main stability result now follows directly from Proposition \ref{prop:absorblogderiv} together with Corollary \ref{cor:firststabgabor}.
\begin{theorem}\label{thm:gaborstability}
 Let $\Omega \subset \R^{2d}$, let $1\le p < 1+1/(2d-1)$ and $q>p/(1-p\frac{2d-1}{2d})$.
 Then for all $f\in \mathcal{M}^p(\R^{2d})$ whose spectrogram $|\mathcal{G}f|$ has a global maximum at $z_0$ it holds that 
 \begin{multline*}
  \inf_{|a|=1}\|\mathcal{G}g-a\mathcal{G}f\|_{L^p(\Omega)} \lesssim 
 (1+h(|\mathcal{G}f|^p,\Omega)^{-1})\cdot\\ 
  \left(\||\mathcal{G}f|-|\mathcal{G}g|\|_{W^{1,p}(\Omega)} + \|(1+|\cdot-z_0|^{2d+2}) \left(|\mathcal{G}f|-|\mathcal{G}g|\right)\|_{L^q(\Omega)}\right), \quad \forall g\in \mathcal{M}^p(\R^d),
 \end{multline*}
where the implicit constant depends on $d,p$ and $q$ only.
\end{theorem}
\subsection{Noise stability}\label{subsec:noisystab}
In virtually any practical situation the measurements are corrupted by noise, i.e., one is faced with the problem of reconstructing $f$ from $|\mathcal{G}f|+\gamma$ instead of $|\mathcal{G}f|$.
As we will see next, the main theorem of the previous section also implies a stability result for reconstruction from noisy Gabor magnitudes.
\begin{theorem}
 Let $\Omega\subset \R^{2d}$, let $1\le p < 1+1/(2d-1)$ and $q>p/(1-p\frac{2d-1}{2d})$.
 Suppose that $f\in\mathcal{M}^p(\R^{2d})$ is such that its spectrogram $|\mathcal{G}f|$ has a global maximum at $z_0$. 
 Suppose that $\gamma$ is a smooth function on $\Omega$ and suppose that $g\in \mathcal{M}^p(\R^d)$ is such that
 $$
 \||\mathcal{G}f|+\gamma-|\mathcal{G}g|\|_{\mathcal{D}}\le \epsilon
 $$
 where
 $$
 \|F\|_{\mathcal{D}}:= \|F\|_{W^{1,p}(\Omega)} + \|(1+|\cdot-z_0|^{2d+2}) F\|_{L^q(\Omega)}.
 $$
 Then it holds that
 $$
 \inf_{|a|=1} \|\mathcal{G}g-a\mathcal{G}f\|_{L^p(\Omega)} \lesssim (1+h(|\mathcal{G}f|^p,\Omega)^{-1}) (\epsilon + \|\gamma\|_{\mathcal{D}}),
 $$
 where the implicit constant depends on $d,p$ and $q$ only.
\end{theorem}
\begin{proof}
 By Theorem \ref{thm:gaborstability} we have that
 $$
 \inf_{|a|=1} \|\mathcal{G}g-a\mathcal{G}f\|_{L^p(\Omega)} \lesssim (1+h(|F_1|^p,\Omega)^{-1}) \||\mathcal{G}g|-|\mathcal{G}f|\|_{\mathcal{D}}.
 $$
 The statement then follows from the estimate
 $$
 \||\mathcal{G}g|-|\mathcal{G}f|\|_{\mathcal{D}} \le \||\mathcal{G}g|-(|\mathcal{G}f|+\gamma)\|_{\mathcal{D}} + \|\gamma\|_{\mathcal{D}} \le \epsilon + \|\gamma\|_{\mathcal{D}}.
 $$
\end{proof}
\subsection{Multicomponent stability}\label{subsec:multicompstab}
In this section we discuss yet another consequence of the main stability result, Theorem \ref{thm:gaborstability}, 
which tells us that instabilities for Gabor phase retrieval must be of disconnected type.
In other words reconstruction of the Gabor transform is stable on domains $\Omega$ where $|\mathcal{G}f|$ is connected.\\
We now want to pick up on the multicomponent paradigm, which was introduced in earlier work by one of the authors and his collaborators\cite{alaifari2016stable}:
Suppose that the phase retrieval problem is relaxed as we require no longer that $\mathcal{G}f$ is to be reconstructed up to a global phase factor 
but instead only demand that the phase factor is constant on each component but may take different values on different components.
A component is here a subdomain $\Omega_i$ of $\Omega$ on which $|\mathcal{G}f|$ is connected, i.e., stable recovery on $\Omega_i$ is possible.\\
The multicomponent paradigm -- i.e. to identify $F=\sum_{i=1}^k F_i$, where $F_i$ is concentrated on $\Omega_i$ with $\sum_{i=1}^k a_i F_i$ whenever $|a_1|,\ldots, |a_k|=1$ --
is especially meaningful for applications in audio as a change of phase on individual components is usually imperceptible for the human ear.\\
The relaxation accomplishes that Gabor phase retrieval becomes stable.
By applying Theorem \ref{thm:gaborstability} on every single component $\Omega_i\subset \Omega$ we obtain the following result. 
\begin{theorem}\label{thm:multicompstab}
 Let $\Omega\subset \R^{2d}$, let $1\le p < 1+1/(2d-1)$ and $q>p/(1-p\frac{2d-1}{2d})$.
 Suppose that $f\in\mathcal{M}^p(\R^{d})$ is such that its spectrogram $|\mathcal{G}f|$ has a global maximum at $z_0$. 
 Suppose that $\Omega$ is partitioned into subdomains $\Omega_1,\Omega_2,\ldots,\Omega_k$, i.e.,
 $\Omega_i\cap \Omega_j = \emptyset$, $\forall i\neq j$ and $\bigcup_{i=1}^k \overline{\Omega_i} = \overline{\Omega}$.\\
 Then it holds for all $g\in\mathcal{M}^p(\R^{d})$ that 
 \begin{multline*}
 \inf_{|a_1|,\ldots,|a_k|=1} \sum_{i=1}^k \|\mathcal{G}g-a_i \mathcal{G}f\|_{L^p(\Omega_i)} \\
 \lesssim (1+h^*) \left(\||\mathcal{G}f|-|\mathcal{G}g|\|_{W^{1,p}(\Omega)} + \|(1+|\cdot-z_0|^{2d+2}) \left(|\mathcal{G}f|-|\mathcal{G}g|\right)\|_{L^q(\Omega)}\right),
 \end{multline*}
 where $h^*:=\max_{1\le i \le k} h(|\mathcal{G}f|^p,\Omega_i)^{-1}$ and where the implicit constant depends on $d,p$ and $q$ only.
\end{theorem}

\section*{Acknowledgements} PG would like to thank Wilhelm Schlag for inspiring discussions and for suggesting the problem of extending the results of \cite{stablegaborpr} to the multivariate case. This research has been supported by the Austrian Science Fund (FWF), grant P-30148. 
\bibliographystyle{abbrv}
\bibliography{phasecluster_CPAM}
\end{document}